\documentclass[paper=a4, fontsize=11pt]{scrartcl} % A4 paper and 11pt font size
\usepackage{microtype}
\usepackage[T1]{fontenc} % Use 8-bit encoding that has 256 glyphs
\usepackage{amsmath,amsfonts,amsthm,amssymb} % Math packages
\usepackage[american]{babel}
\usepackage{csquotes}
\usepackage{tikz}
\usepackage[maxbibnames=99,style=alphabetic,sortcites=true,sorting=nyt,citestyle=alphabetic, backend=biber]{biblatex}
\DeclareLanguageMapping{american}{american-apa}%bibliography package
\usepackage{sectsty} % Allows customizing section commands
\allsectionsfont{\centering \normalfont\scshape} % Make all sections centered, the default font and small caps

\usepackage{scrlayer-scrpage}
\usepackage[normalem]{ulem}

\deftriplepagestyle{fancyplain}
    {}
    {}
    {}
    {}
    {}
    {\pagemark}
\usepackage{sectsty}
\sectionfont{\bfseries\raggedright}
\subsectionfont{\normalfont\itshape\raggedright}

\usepackage{mathtools}
\usepackage{tikz-cd}
\usepackage{outlines}
\usepackage{enumitem}
\numberwithin{equation}{section} % Number equations within sections (i.e. 1.1, 1.2, 2.1, 2.2 instead of 1, 2, 3, 4)
\numberwithin{figure}{section} % Number figures within sections (i.e. 1.1, 1.2, 2.1, 2.2 instead of 1, 2, 3, 4)
\numberwithin{table}{section} % Number tables within sections (i.e. 1.1, 1.2, 2.1, 2.2 instead of 1, 2, 3, 4)

\bibliography{biblio}

\newtheorem{thm}{Theorem}[section]
\newtheorem{lemma}[thm]{Lemma}
\newtheorem{conj}[thm]{Conjecture}
\newtheorem{cor}[thm]{Corollary}
\newtheorem{prop}[thm]{Proposition}

\theoremstyle{definition}
\newtheorem{definition}[thm]{Definition}

\theoremstyle{remark}
\newtheorem*{remark}{Remark}
\setenumerate[1]{label=\Roman*.}
\setenumerate[2]{label=\Alph*.}
\setenumerate[3]{label=\roman*.}
\setenumerate[4]{label=\alph*.}

\def\BB{\mathbb{B}}

\def\Proj{\mathbb{P}}

\def\Aff{\mathbb{A}}

\let\QQ\Rational

\let\ZZ\Integer

\def\calO{\mathcal{O}}

\def\gcd{\operatorname{gcd}}

\def\genus{\operatorname{genus}}

\def\M{\text{M}_{11}}

\def\PSL{\text{PSL}_2(11)}
\def\Ru{\text{Ru}}

\newcommand\Ind[2]{{
  \text{Ind}_{#2}^{#1}
  }}

\newcommand{\Addresses}{{
  \bigskip
  \textsc{Dean Bisogno: Colorado State University, Fort Collins, CO 80523, USA}\par\nopagebreak
  \textit{E-mail address}: \texttt{bisogno@math.colostate.edu}

}}

\title{
\normalfont \normalsize
\huge Abhyankar's Inertia Conjecture for Some Sporadic Groups\\
}

\author{Dean Bisogno}

\date{\normalsize\today}

\begin{document}

\maketitle

\begin{abstract}
\noindent ABSTRACT. This paper verifies Abhyankar's Inertia Conjecture for certain sporadic groups $G$ in particular characteristics by showing that all possible inertia groups occur for $G$-Galois covers of the affine line. For a larger set of sporadic groups, all but finitely many possible ramification invariants are shown to occur. In particular, we prove that all but eight of the possible ramification invariants are realizable for $\text{M}_{11}$-Galois covers of the affine line.


\end{abstract}

\section{Introduction}
Following the work of Serre, \cite{MR1071640}, Raynaud, and Harbater proved Abhyankar's Conjecture for Galois covers of affine curves in positive characteristic. Let $k$ be an algebraically closed field of characteristic $p$. Let $G$ be a finite group and $p(G)$ be the normal subgroup of $G$ generated by elements of $p$-power order.
\begin{thm}[Abhyankar's Conjecture \cite{MR0094354,MR1253200,MR1269423}]\label{thm:AC}
Let $X$ be a smooth projective curve of genus $g$ defined over $k$. Let $B$ be a finite
non-empty set of points of $X$ having cardinality $r$ and let $U = X \setminus B$. A finite group $G$ is the Galois group of an unramified cover of $U$ if and only if $G/p(G)$ has a generating set of size at most $2g + r - 1$.
\end{thm}
Call $G$ quasi-$p$ if $G = p(G)$. A simple group is quasi-$p$ for any prime dividing its order. When $X$ is the projective line $\Proj_k^1$ and $B = \{ \infty \}$, then Theorem~\ref{thm:AC} states that a finite group $G$ is the Galois group of an unramified cover of $\Aff_k^1$ if and only if a generating set of $G/p(G)$ has size at most $0$. Thus a finite group $G$ is the Galois group of an unramified cover of the affine line over $k$ if and only if $G$ is quasi-$p$. Following the proof of Theorem~\ref{thm:AC}, Abhyankar stated Conjecture~\ref{AIC}.

\begin{conj}[Abhyankar's Inertia Conjecture {\cite[Section 16]{MR1816069}}]\label{AIC}
Let $G$ be a finite quasi-$p$ group. Let $I$ be a subgroup of $G$ which is an extension of a cyclic group of order prime-to-$p$ by a $p$-group $J$. Then $I$ occurs as an inertia group for a $G$-Galois cover of $\Proj_k^1$ branched only at $\infty$ if and only if the conjugates of $J$ generate $G$.
\end{conj}

The condition on $J$ in Conjecture~\ref{AIC} is necessary. Suppose $G$ and $I$ are as in Conjecture~\ref{AIC} and that $I$ is the inertia group of some $G$-Galois cover of $\Proj_k^1$ branched only at $\infty$. Let $H$ be the normal subgroup of $G$ generated by the conjugates of $J$. Then the $G/H$-Galois quotient cover is tamely ramified at $\infty$. Grothendieck showed that the tame fundamental group of the affine line is trivial \cite[Corollary~XIII.2.12]{MR0217087}. Consequently, $H = G$ which proves the ``only if'' direction of Conjecture~\ref{AIC}.
%Also note that by the Schur-Zassenhaus Theorem \cite[pg.\ 132]{MR0091275}, the extension of a cyclic group of order $m$ coprime to $p$ by a $p$-group of order $m$ is a semi-direct product.

Fix $k = \overline{F}_p$ and a quasi-$p$ group $G$. 
\begin{definition}
Denote the set of potential inertia groups of $G$-Galois covers of $\mathbb{P}^1_k$ branched only at $\infty$ by $\mathcal{I}_p(G)$. Explicitly $\mathcal{I}_p(G)$ is defined in the following way
\begin{align*}
\mathcal{I}_p(G) = \{I \subset G \mid I \text{ satisfies the hypotheses of Conjecture~\ref{AIC}}\}.
\end{align*}
Throughout this paper we specify a $G$-Galois cover of $\mathbb{P}^1_k$ branched only at $\infty$ with particular inertia group $I \in \mathcal{I}_p(G)$ at a ramified point. Such a cover is called a $(G,I)$-Galois cover. We say that Conjecture \ref{AIC} is true (or verified) for $G$ in characteristic $p$ if for every $I \in \mathcal{I}_p(G)$ there exists a $(G,I)$-Galois cover.
\end{definition}

This paper verifies Conjecture \ref{AIC} for certain sporadic groups in various characteristics. In order to do so we prove Lemma~\ref{lma:BijectionAndLabels}, a technical lemma which allows us to construct a well-defined thickening problem. Work of Habater and Stevenson \cite{MR1670658} and Pries \cite{MR2016596} determines the existence of solutions to these thickening problems. This allows us to prove the following theorem.
\begin{thm}
Suppose $H \subset G$ are finite quasi-$p$ groups, the index $[G:H]$ is coprime to $p$, a Sylow $p$-subgroup of $G$ has order $p$, and every $I \in \mathcal{I}_p(G)$ is a $G$-conjugate of some $I' \in \mathcal{I}_p(H)$. If Conjecture~\ref{AIC} is true for $H$ in characteristic $p$, then it is true for $G$ in characteristic $p$.
\end{thm}
As an application of the previous theorem we consider sporadic groups with stipulated properties.
\begin{itemize}[noitemsep,topsep=\parskip]
  \item Sylow $p$-subgroups of $G$ are isomorphic to $\mathbb{Z}/p$.
  \item The normalizer $\text{N}_G(S)$ is isomorphic to $\mathbb{Z}/p \rtimes \mathbb{Z}/((p-1)/2)$.
  \item The group $G$ contains a subgroup isomorphic to $\text{PSL}_2(p)$.
\end{itemize}
These attributes are sufficient to verify Conjecture \ref{AIC}.
\begin{cor}
Abhyankar's Inertia Conjecture is true for the fourteen sporadic groups and characteristics in Table \ref{tbl:aicgroups}.
\end{cor}

The ramification invariant of a cover is an invariant of the filtration of higher ramification groups in the upper numbering. The ramification invariant is necessary though not sufficient to determine the genus of the covering curve associated to a $(G,I)$-Galois cover. More information can be found in Section \ref{sec:ramgrps}.

In Section \ref{sec:infConductors}, we study the ramification invariants that can occur for $G$-Galois covers of $\Proj_k^1$ branched only at $\infty$ when $G$ contains a subgroup $H \cong \text{PSL}_2(p)$. In Section \ref{sec:mathieu11} we verify a refinement of Conjecture \ref{AIC} for the Mathieu group $\M$: all but eight of the possible ramification invariants occur for $\M$-Galois covers of $\Proj_k^1$ branched only at $\infty$ in characteristic 11. We leave it as an open question whether these eight occur as well.

\begin{thm}
Conjecture~\ref{AIC} is true for $\M$ in characteristic $p=11$. Further, all possible ramification invariants are verified to occur except for those in the set $\{6/5, 7/5, 9/5, 12/5, 14/5, 17/5, 19/5, 27/5\}$.
\end{thm}

We prove similar result for additional sporadic groups in Theorem~\ref{thm:infjumps}. 

Previous work has been successful when considering simple groups which are not sporadic. In \cite[Section 4.1]{MR1403966} and \cite[Theorem~2]{MR1221836}, Harbater shows that the Sylow $p$-subgroups of the Galois group occur as inertia groups. Abhyankar's Inertia Conjecture (Conjecture \ref{AIC}) is true for the following groups:
\begin{enumerate}[label=\alph*)]
\item $\text{PSL}_2(p)$ for $p \geq 5$, \cite[Corollary~3.3]{MR2003452};
\item $\text{A}_p$ for $p \geq 5$, \cite[Corollary~3.5]{MR2003452};
\item $A_{p+2}$ when $p$ is odd and $p \equiv 2 \text{ mod } 3$ \cite[Theorem~1.2]{MR2891701}.
\end{enumerate}

In \cite{MR3203555}, Obus shows inertia groups isomorphic to $\ZZ/p^r$ and $D_{p^r}$ are realizable for $\text{PSL}_2(l)$ in characteristic $p$ when $p^m$ divides $|\text{PSL}_2(l)|$, $l \not= p$ is an odd prime and $1\leq r \leq m$. Das and Kumar show that certain inertia groups occur for covers whose Galois group is a product of alternating groups \cite[Corollary~4.9]{2017arXiv171107756D}. Refined observations are made in both \cite{MR2003452} and \cite{MR2891701} beyond just the verification of Conjecture \ref{AIC}. Both papers are able to determine that all but finitely many of the possible ramification invariants occur. Further reading can be found in \cite[Section 4]{MR2891701}.

\textbf{Acknowledgements.} Thank you to Dr.\ Rachel Pries for suggesting this direction of inquiry. The author would also like to thank Dr.\ Jeff Achter, Dr.\ Renzo Cavalieri, and Dr.\ Alexander Hulpke.

\section{Preliminaries}
\subsection{Ramification groups}\label{sec:ramgrps}
Let $\phi \colon X \to Y$ be a $G$-Galois cover of curves with $\xi$ a point of $Y$ and $\eta$ a point in the fiber over $\xi$. Let $\calO_\eta$ denote the discrete valuation ring of $\calO_Y$ given by the valuation $\nu_\eta$ at $\eta$. For $i\geq-1$, the $i^{\text{th}}$ ramification group is given by
\begin{equation}\label{eqn:HigherRamification}
G_i = \{\delta \in G : \nu_\eta(\delta(a)-a) \geq i+1 \text{ for all } a\in \calO_\eta\}.
\end{equation}
The higher ramification groups form a filtration
\begin{equation}\label{eqn:flitration}
\{G_i\}_{i\geq -1} : G_{-1} \supseteq G_0 \supseteq G_1 \supseteq \ldots.
\end{equation}

The subgroup $G_{-1}$ is the decomposition group $D_{\eta}$ at $\eta$. It is the subgroup of $G$ of automorphisms that fix $\eta$. The inertia group $I_\eta$ at $\eta$ is $G_{0}$. In general, if $\pi$ is a uniformizer of $\calO_\eta$, then $G_i$ is the kernel of the action of $G_{-1}$ on $\calO_\eta/\pi^{i+1}$. The subscript $\eta$ on inertia and decomposition groups is suppressed unless relevent.

The ordering of the ramification groups in \eqref{eqn:flitration} is called the lower numbering while the renumbering introduced in Definition~\ref{defn:LowerNumbering} is called the upper numbering.
\begin{definition}[Upper Numbering {\cite[Section IV.iii]{MR554237}}]\label{defn:LowerNumbering}
Consider the function
\begin{align*}
t = \operatorname{H}(s) = \int_{0}^{s} \frac{dx}{[G_0 : G_x]},
\end{align*}
called the Herbrand function and let $\psi(t)$ be the inverse map of $H(s)$.
Then for any real $s \geq -1$, let $G_s = G_{\lceil s \rceil}$ and renumber the ramification groups by $G^t = G_s$.
\end{definition}
\begin{definition}[Jumps]
An index $t$ such that $G^{t} \not=G^{t+\epsilon}$ for any $\epsilon > 0$ is called an upper jump.
\begin{enumerate}[label=\alph*)]
\item The largest upper jump $\sigma$ is called the ramification invariant.
\item Let $j = \psi(\sigma)$. This is called the inertia jump; it is the index of the last nontrivial ramification group in the lower numbering.
\end{enumerate}
\end{definition}

Let $\phi \colon X \to \Proj_k^1$ be a $(G,I)$-Galois cover for some $I\in \mathcal{I}_p(G)$ and $\eta$ a ramified point with inertia group $I$. We denote the normalizer in $G$ of a subgroup $I\subset G$ by $\text{N}_G(I)$. The inertia groups at other ramification points are all the $G$-conjugates of $I$ of which there are $[G:\text{N}_G(I)]$. For every $G$-conjugate $I'$ of $I$, the number of ramified points with inertia group $I'$ is $[\text{N}_G(I):I]$. If a particular group structure is specified for $I$, it is meant that the inertia groups of $\phi$ are subgroups of $G$ isomorphic to $I$.

If $p$ strictly divides $|I|$, then $I$ is a semi-direct product of the form $\ZZ/p \rtimes \ZZ/m_I$ where $\gcd(p,m_I)=1$ by the Schur-Zassenhaus Theorem \cite[pg.\ 132]{MR0091275}. In this case, there is exactly one inertia jump $j$ and $p\nmid j$. The ramification invariant is then related to the inertia jump by $\sigma = j/m_I$.

The following proposition provides some restrictions on the inertia jump and possible inertia groups.
\begin{prop}[{\cite[Proposition~IV.ii.9]{MR554237}}]
Suppose $\phi$ is a $(G,I)$-Galois cover with inertia jump $j$. By \cite[Corollary~IV.ii.4]{MR554237}, $I$ is an extension of a cyclic group $C$ of order $m$ by a $p$-group $P$ via a group homomorphism $\psi\colon C \hookrightarrow \text{\emph{Aut}}(P)$. If $\tau\in I$ with order $p$ and $\beta \in I$ with order $m$, then 
\begin{align*}
\psi(\beta) \tau \psi(\beta^{-1}) = \psi(\beta)^j \tau.
\end{align*}
\end{prop}

\subsection{$p$-Properties of Galois groups}

Recall from Theorem~\ref{thm:AC} that the existence of $G$-Galois covers of $\mathbb{P}^1$ branched only at $\infty$ in characteristic~$p > 0$ is detected by the quasi-$p$ condition on $G$.
\begin{definition}[quasi-$p$]\label{defn:quasip}
Denote by $p(G)$ the subgroup of $G$ generated by all $p$-power elements of $G$. If $p(G) = G$, then call $G$ quasi-$p$.
\end{definition}

All pairs $G$ and $p$ which we study in this paper are chosen such that $G$ is simple and $p$ divides $|G|$.
\begin{lemma}\label{lma:SimpleQuasip}
If $G$ is simple and $p$ divides the order of $G$, then $G$ is quasi-$p$.
\end{lemma}
\begin{proof}
The subgroup $p(G)$ is normal and non-trivial in $G$. By the hypothesis, $G$ is simple and thus satisfies $p(G)=G$.
\end{proof}

The following condition, $p$-pure, on $G$ was introduced by Raynaud. It is a geometric condition that guarantees that the reduction of a $G$-Galois cover of the affine line is connected over a terminal component. More techniques are available for $p$-pure groups see \cite{MR1253200} and \cite{MR1937118} for details.
\begin{definition}[$p$-pure {\cite[pg.\ 426]{MR1253200}}]\label{defn:pPurity}
Let $G$ be a finite quasi-$p$ group and let $S$ be a fixed Sylow $p$-subgroup of $G$. By $G(S)$ denote the subgroup of $G$ generated by all proper, quasi-$p$ subgroups $H \subset G$ having a Sylow $p$-subgroup contained in $S$. If $G(S)\not=G$ then $G$ is $p$-pure.
\end{definition}

\begin{definition}[$p$-weight {\cite[Definition 3.1.2]{MR2016596}}]
Fix $G$ and $S$ as in Definition~\ref{defn:pPurity}. Consider all subgroups $G^\prime \subset G$ such that $G^\prime$ is quasi-$p$ and $p$-pure such that $G^\prime \cap S$ is a Sylow $p$-subgroup of $G^\prime$. The $p$-weight $\omega_G$ of $G$ is the minimal number of such subgroups $G^\prime$ of $G$ which are needed to generate $G$. Note that a group $G$ is $p$-pure if $\omega_G = 1$.
\end{definition}

\subsection{Sporadic groups}\label{sec:mathieugroups}
The Mathieu groups $\text{M}_{11}$, $\text{M}_{12}$, $\text{M}_{22}$, $\text{M}_{23}$, and $\text{M}_{24}$ are sporadic simple groups first described by \'Emile Mathieu in the 1870s \cite[pg.\ 389]{MR514842}. The group $\M$ has order $7920=2^4\cdot3^2\cdot5\cdot11$ and acts strictly 4-transitively on 11 objects. By \cite[pg.\ 18]{MR827219}, there are two 11-conjugacy classes labeled 11a and 11b. Conjugate maximal subgroups of $\M$ are the following \cite[pg.\ 18]{MR827219}.
\begin{table}[htbp]
\small
\centering
  \begin{tabular}{ | c | c | c | c | c | c |}
    \hline
    Subgroup & $\text{M}_{10}$ & $\PSL$ & $\text{M}_9:2$ & $S_5$ & $Q:S_3$\\ \hline
    Order & 720 & 660 & 144 & 120 & 48 \\
    \hline
  \end{tabular}
  \caption{Maximal Subgroups of $\M$}
  \label{tbl:subgroups}
\end{table}

\begin{lemma}
The groups $\M$ and $\text{M}_{22}$ are quasi-$11$ and $11$-pure.
\end{lemma}
\begin{proof}
For $11$-purity, see Lemma~\ref{lma:SimpleQuasip}.

To check $11$-purity, pick $G\in\{\text{M}_{11},\text{M}_{22}\}$. Fix a Sylow 11-subgroup $S$ of $G$. The only quasi-11 subgroups containing $S$ are its normalizer $\text{N}_{G}(S)$ and a unique subgroup $T$ isomorphic to $\PSL$. But $\text{N}_{G}(S)\nobreak \subset\nobreak T$; thus $G$ is 11-pure.
\end{proof}

\begin{remark}
The groups $\text{M}_{12}$, $\text{M}_{23}$, and $\text{M}_{24}$ are not 11-pure. For $G \cong \text{M}_{12}$ every Sylow 11-subgroup of $G$ is contained in both a maximal subgroup $H \cong \PSL$ of $G$ and a maximal $K \cong \M$ of $G$. The groups $H$ and $K$ are maximal subgroups, consequently $H \not\subset K$. Hence $G(S) = G$ and $\text{M}_{12}$ is not $11$-pure. This argument works similarly for $\text{M}_{23}$ and $\text{M}_{24}$. Likewise, $\text{M}_{22}$ is not $7$-pure and $\text{M}_{24}$ is not $23$-pure.
\end{remark}

Both the Higman-Sims group HS and McLaughlin group McL are stabilizers of certain planes in the Leech Lattice. The group HS stabilizes the plane given by the 3-3-2 triangle. The group McL stabilizes the plane given by the 3-2-2 triangle. The groups HS and McL have order strictly divisible by 11, have Sylow 11-subgroups isomorphic to $\ZZ/11$ with normalizers isomorphic to $\ZZ/11 \rtimes \ZZ/5$, and contain a subgroup isomorphic to $\PSL$. Both $\text{HS}$ and $\text{McL}$ fail to be $11$-pure.

\begin{table}[htbp]
\centering
\small
\begin{tabular}{| l | l | l |}
\hline
Group & Order & Reference \\ \hline
$\text{HS}$   & $2^9 3^2 5^3\cdot7\cdot11$          & \cite{MR0227269}  \\ \hline
$\text{McL}$  & $2^7 3^6 5^3\cdot7\cdot11$          & \cite{MR0242941}  \\ \hline
$\Ru$         & $2^{14}3^{3}5^{3}\cdot7\cdot 13\cdot 29$ & \cite{MR727376}   \\ \hline
\end{tabular}
\end{table}

The group $\text{Ru}$ has Sylow $29$-subgroups isomorphic to $\mathbb{Z}/29$ with normalizers isomorphic to $\mathbb{Z}/29 \rtimes \mathbb{Z}/14$, and contains a maximal subgroup isomorphic to $\text{PSL}_2(29)$. Further, this is the only maximal subgroup of $\text{Ru}$ with order divisible by $29$. Consequently $\text{Ru}$ is 29-pure.

\section{Resolving Abhyankar's Inertia Conjecture from Subgroups}\label{sec:fromsubgrps}
Few techniques are known to increase the size of inertia groups. A technique we demonstrate in this section constructs thickening problems which have solutions which are known to exist by results of Harbater and Stevenson \cite[Theorem~4]{MR1670658}. In \cite{MR2016596} it is shown that inertia groups and ramification invariants behave predictably under this operation.

\subsection{A Galois equivariant relation on ramification points}\label{sec:equivrel}
We begin by fixing some notation. Fix a $(G,I)$-Galois cover $\phi : X \to \Proj_k^1$. Pick a ramified point $\eta$ on $X$ and denote the inertia group at $\eta$ by $I_{\eta}$. The group $G$ acts transitively on ramification points, thus for each ramification point $\epsilon$ there exists a $g \in G$ such that $g\circ \eta =\epsilon$. Let $I_g$ denote the inertia group at the ramified point $g \circ \eta$, consequently $I_g  = g I_\eta g^{-1}$.

Note that $g_1 \circ \eta = g_2 \circ \eta$ if and only if $g_2^{-1}g_1 \in I_\eta$. This is because $k$ is algebraically closed so the decomposition group at $\eta$ and $I_\eta$ coincide.

We define an equivalence relation on ramification points.
\begin{definition}
We say $g_1 \circ \eta \sim g_2 \circ \eta$ if and only if $g_2^{-1} g_1 \in \text{N}_{G}(I_\eta)$. In particular this identifies $\eta$ with the ramification points $z \circ \eta$ for all $z \in \text{N}_{G}(I_\eta)$. 
\end{definition}

\begin{lemma}\label{lma:normalizers} 
Suppose $p$ divides the order of $G$ and $I_\eta \in \text{Syl}_p(G)$. The groups $$\text{N}_{G}(I_{g_1})=\text{N}_{G}(I_{g_2})$$ as subgroups of $G$ if and only if $g_1 \circ \eta \sim g_2 \circ \eta$.
\end{lemma}

\begin{proof}
First we show that $\text{N}_G(I_{g_1}) = \text{N}_G(I_{g_2})$ if and only if $I_1 = I_2$. Assume $\text{N}_G(I_{g_1}) = \text{N}_G(I_{g_2})$. By the Sylow theorems, $\text{N}_G(I_{g_i})$ contains a unique Sylow $p$-subgroup. Both $I_{g_1}$ and $I_{g_2}$ are the Sylow $p$-subgroup of $\text{N}_G(I_{g_1})$. This shows that $I_{g_1} = I_{g_2}$ as subgroups of $G$. Alternatively if $I_{g_1} = I_{g_2}$ as subgroups of $G$, then the normalizers $\text{N}_G(I_{g_1})$ and $\text{N}_G(I_{g_2})$ must be equal as well.

Consequently, we must show that $g_2^{-1} g_1 \in \text{N}_G(I_\eta)$ if and only if $I_1 = I_2$ as subgroups of $G$. We proceed by computing
\begin{align*}
g_2^{-1} g_1 \in \text{N}_G(I_\eta) &\iff I_\eta = g_2^{-1}g_1 I_\eta g_1^{-1}g_2 \\
                                    &\iff g_2 I_\eta g_2^{-1} = g_1 I_\eta g_1^{-1} \\
                                    &\iff I_{g_2} = I_{g_1}.
\end{align*}
\end{proof}

\begin{cor}\label{cor:simSize}
The relation $\sim$ collects the ramification points of $\phi$ into equivalence classes of cardinality $[\text{N}_G(I_{\eta}) : I_{\eta}]$ identified by subgroups of $G$ isomorphic to $\text{N}_G(I_{\eta})$.
\end{cor}

\begin{proof}
This follows immediately from Lemma~\ref{lma:normalizers}.
\end{proof}

Suppose $\phi \colon X \to \Proj_k^1$ is a $(G,I)$-Galois cover. The set of ramification points of $\phi$ is denoted by $R_\phi$ and the cardinality of $R_\phi$ is $[G:I]$. The number of points in $R_\phi$ with inertia group precisely $I$ is $[\text{N}_G(I):I]$. The set of equivalence classes of $R_\phi/\sim$ is denoted by $\overline{R}_\phi$ and the cardinality of $\overline{R}_\phi$ is $[G : \text{N}_{G}(I)]$.

\subsection{Induced covers, patching, and deformations}
For the remainder of this section fix a finite quasi-$p$ group $G_1$ and a quasi-$p$ subgroup $G_2$ with index coprime to $p$. Let $S$ be a Sylow $p$-subgroup of $G_1$ and choose $I_i$ containing $S$. Assume that $(G_i, I_i)$-Galois covers $\phi_i \colon X_i \to \Proj_k^1$ exist.

Recall the proof of \cite[Corollary~2.3.1]{MR2016596}. A similar process is implemented here. We will induce a disconnected $(G_1, I_2)$-Galois cover $\varphi_2$ from a $(G_2, I_2)$-Galois cover. The induced cover $\varphi_2$ and a connected $(G_1, I_1)$-Galois cover are formally patched in neighborhoods of the ramification points. This operation yields a $G_1$-Galois thickening problem for which there is a solution $\mathbb{V}$ \cite[Theorem~4]{MR1670658}. Deformations of the special fiber of $\mathbb{V}$ yield a smooth, connected $(G_1,I_2)$-Galois cover.

We extend the notation of Section~\ref{sec:equivrel} to serve two covers. Fix a ramified point $\eta_i$ of $\phi_i$. By $I_{g,i}$ we denote the inertia group at the ramified point $g \circ \eta_i$.
\begin{definition}\label{defn:induced}
Suppose $\phi_2 \colon X_2 \to \mathbb{P}^1_k$ is a $G_2$-Galois cover of curves. The induced curve $\mathcal{X}_2 := \Ind{G_1}{G_2}(X)$ is defined to be the disconnected curve consisting of $[G_1:G_2]$ copies of $X_2$, indexed by left cosets of $G_2$ in $G_1$. There is an induced action of $G_1$ on $\mathcal{X}_2$.
The induced cover is denoted $\varphi := \Ind{G_1}{G_2}(\phi_2)\colon \mathcal{X}_2 \to \mathbb{P}^1_k$.
\end{definition}

\begin{lemma}\label{lma:BijectionAndLabels}
For each $i \in \{ 1,2 \}$ let $\phi_i \colon X_i \to \Proj_k^1$ be a $(G_i,I_i)$-Galois cover. Suppose $G_2 \subset G_1$ and let $\varphi_2=\text{\emph{Ind}}_{G_2}^{G_1}(\phi_2)$ be the induced cover. If $\text{N}_{G_1}(I_1) \cong \text{N}_{G_1}(I_2)$, then there is a set bijection $b \colon \overline{R}_{\varphi_2} \to \overline{R}_{\phi_1}$. Further, there is a labeling of ramification points such that the bijection $b$ is $G_1$-equivariant. 
\end{lemma}
\begin{proof}
First we check that the cardinalities of $\overline{R}_{\varphi_2}$ and $\overline{R}_{\phi_1}$ agree:
  \begin{align*}
  |\overline{R}_{\varphi_2}| &= [G_1 : G_2]|\overline{R}_{\phi_2}| = [G_1 : G_2][G_2 :\text{N}_{G_2}(I_2)] \\
                              &= [G_1 : \text{N}_{G_1}(I_1)]= |\overline{R}_{\phi_1}|.
  \end{align*}
The equality of the first and second lines is justified by the hypothesis $\text{N}_{G_1}(I_1) \cong \text{N}_{G_1}(I_2)$. 

Applying Corollary~\ref{cor:simSize}, define $b$ to be the bijection sending the equivalence class of $\overline{R}_{\varphi_2}$ identified by $N$ to the corresponding class of $\overline{R}_{\phi_1}$.

We now show that that $b$ is $G_1$-equivariant. Let $b(\eta) \in \overline{R}_{\phi_1}$ be a ramification point with inertia group $I_\eta$ and normalizer of inertia $N$. For any $g \in G_1$, $g\circ \eta$ has inertia group $g I_\eta g^{-1}$. We must show that $g \circ b(\eta)$ has inertia group $g I_\eta g^{-1}$. Recall that by definition $b(\eta)$ has normalizer of inertia $N$. Every ramification point with normalizer $N$ has inertia group $I_\eta$. Consequently, $g \circ b(\eta)$ has inertia group $g I_\eta g^{-1}$.
\end{proof}

\begin{lemma}\label{lma:makingjumps}
Suppose $p$ is prime and $G$ is a finite quasi-$p$ group with order strictly divisible by $p$. Fix a quasi-$p$ subgroup $H \subset G$, and $I\in\mathcal{I}_p(H)$ with $I \cong \mathbb{Z}/p \rtimes \mathbb{Z}/m_I$. If there exists an $(H,I)$-Galois cover with inertia jump $j$, then there exists an $(H,I)$-Galois cover with inertia jump $j+ i m_I$ and a $G$-Galois cover with inertia jump $\gamma(j+ i m_I)$ for some positive integers $i$ and $\gamma$. 
\end{lemma}
\begin{proof}
Let $S$ be a Sylow $p$-subgroup of $H$ and $G$. There exists a $(G,S)$-Galois cover $\phi$ by \cite[Theorem 2]{MR1221836}. Note that $\phi$ can be selected such that its inertia jump is $\gamma(j+i m_I)$ for some pair of positive integers $i$ and $\gamma$ where $\gcd(\gamma,m_I)=1$; this is a consequence of \cite[Theorem~3.2.4]{MR2016596}. 

By assumption, there exists an $(H,I)$-Galois cover $\psi$ with inertia jump $j$. The inertia jump of $\psi$ is increased to $j+im_I$ which finishes the proof \cite[Theorem~2.2.2]{MR2016596}.
% Without loss of generality, these choices can be made to fulfill the numerical conditions sufficient to apply the Different Inertia case of \cite[Corollary~2.3.1]{MR2016596} to $\psi$ and $\phi$. The result is a $(G,I)$-Galois cover with inertia jump $j_2(\gamma m_I + 1) \equiv j_2 \bmod m_I$ and ramification invariant $\sigma = \frac{j_2(\gamma m_I + 1)}{m_I}$.
\end{proof}

The proof of Theorem~\ref{thm:FromSubgroups} uses formal patching to solve a particular thickening problem. The pattern of proof follows \cite[Theorem~2.3.7]{MR2016596} which uses \cite[Theorem~4]{MR1670658} to ensure a solution exists.

\begin{thm}\label{thm:FromSubgroups}
Consider finite quasi-$p$ groups $G_2 \subset G_1$. Suppose the Sylow $p$-subgroups of $G_1$ have order $p$, fix $I\in\mathcal{I}_p(G_2)$. If there exists a $(G_2, I)$-Galois cover, then there exists a $(G_1, I)$-Galois cover.
\end{thm}
\begin{proof}
Fix a Sylow $p$-subgroup $S$ of $G_1$ contained in $I$. Let $\phi_1 \colon X_1 \to \mathbb{P}_k^1$ be a $(G_1, S)$-Galois cover which exists by \cite[Theorem 2]{MR1221836}. Let $\phi_2$ be a $(G_2,I)$-Galois cover, and $\varphi_2 \colon \mathcal{X}_2 \to \mathbb{P}_k^1$ denote the induced cover. Finally, let $W$ be a curve isomorphic to two $\mathbb{P}^1_k$'s intersecting transversely at $\infty$. Construct $\vartheta: V \to W$ by patching $X_1$ and $\mathcal{X}_2$ at the ramification points identified by the bijection produced in Lemma~\ref{lma:BijectionAndLabels}.

We apply \cite[Theorem~2.3.7]{MR2016596} to $\phi_1$ and $\phi_2$. It is necessary that $|S| = p$ as well as certain numerical conditions are verified for the jumps of $\phi_1$ and $\phi_2$. These numerical conditions can be satisfied by Lemma~\ref{lma:makingjumps}. See \cite[Notation~2.3.2, Notation~2.3.6]{MR2016596} for additional details.

Let $R = k[[t]]$. The result of applying \cite[Theorem~2.3.7]{MR2016596} is the following. A family of covers over an $R$-curve $P_R$ is constructed. The generic fiber of this family is a $(G,I)$-Galois cover, thus deformations of the special fiber yield the result.
\end{proof}

\begin{cor}\label{cor:AICInduced}
Suppose $G_2 \subset G_1$ are finite quasi-$p$ groups, the index $[G_1:G_2]$ is coprime to $p$, and the Sylow $p$-subgroups of $G_1$ have order $p$. Also suppose every $I \in \mathcal{I}_p(G_1)$ is a $G_1$-conjugate of some $I' \in \mathcal{I}_p(G_2)$. If Conjecture~\ref{AIC} is true for $G_2$ in characteristic $p$, then it is true for $G_1$ in characteristic $p$.
\end{cor}

\begin{proof}
Pick $I \in \mathcal{I}_p(G_1)$. By assumption, every element $I \in \mathcal{I}_p(G_1)$ is represented by a $G_1$-conjugate element $I' \in \mathcal{I}_p(G_2)$. Because Conjecture~\ref{AIC} is true for $G_2$, there exists a $(G_2, I')$-Galois cover. Applying Theorem~\ref{thm:FromSubgroups} constructs a $(G_1, I')$-Galois cover $\phi$. The group $G_1$ acts transitively on fibers of $\phi$. For this reason all $G_1$-conjugates of $I'$ occur as inertia groups at some point over $\infty$. This enables us to conclude that $I$ is the inertia group at some ramified point of $\phi$.
\end{proof}

As an application, Conjecture~\ref{AIC} is verified for several sporadic groups due to Conjecture~\ref{AIC} being known for $\text{PSL}_2(p)$ in characteristic $p \geq 5$ \cite[Corollary~3.3]{MR2003452}.

\begin{cor}Abhyankar's Inertia Conjecture is true for the groups and characteristics in Table \ref{tbl:aicgroups}.
\end{cor}
\begin{table}[!h]
\small
\centering
\begin{tabular}{| l | l |}
\hline
$p$     & Groups                        \\ \hline
5,7       & $\text{M}_{22}$               \\ \hline
11      & $\text{M}_{11}$, $\text{M}_{12}$, $\text{M}_{22}$, $\text{M}_{23}$, $\text{HS}$, $\text{McL}$      \\ \hline
13      & $\text{F}_{22}$, $\text{Suz}$ \\ \hline
17, 19  & $\text{J}_{3}$                \\ \hline
23      & $\text{M}_{24}$               \\ \hline
29      & $\text{Ru}$                   \\ \hline
31      & $\text{ON}$, $\text{B}$       \\ \hline
59, 71  & $\text{M}$                    \\ \hline
\end{tabular}
\caption{Groups and characteristics $p$ for which Conjecture~\ref{AIC} is verified by Corollary~\ref{cor:AICInduced}.}
\label{tbl:aicgroups}
\end{table}
\begin{proof}
Fix $G$ isomorphic to a group in Table~\ref{tbl:aicgroups}, $p\not=5$, and set $m_I = (p-1)/2$.
Abhyankar's Inertia Conjecture is known for $\text{PSL}_2(p)$ by \cite[Corollary~3.3]{MR2003452}. The group $G$ contains a subgroup isomorphic to $\text{PSL}_2(p)$. The normalizers of Sylow $p$-subgroups in $G$ and $\text{PSL}_2(p)$ are isomorphic to $\ZZ/p\rtimes\ZZ/m_I$. Consequently, the hypothesis of Corollary~\ref{cor:AICInduced} are satisfied.

In the case $G \cong \text{M}_{22}$ and $p=5$, the proof is similar. The fundamental difference is that we consider a subgroup isomorphic to $\text{A}_7$, for which Abhyankar's Inertia Conjecture is known \cite[Theorem~1.2]{MR2891701}.
\end{proof}

\begin{remark}
This strategy of proof does not work for $\text{M}_{24}$ with $p=11$ because the normalizer of a Sylow 11-subgroup of $\text{M}_{24}$ has order 110. There is no proper subgroup $H \subset \text{M}_{24}$ for which it is known that there exists an $H$-Galois cover with inertia order 110. Consequently, this method does not verify Conjecture~\ref{AIC}. In the next section we verify the existence of $\text{M}_{24}$-Galois covers of the affine line with all but finitely many potential inertia jumps.
\end{remark}

\subsection{Example: The Monster group $\text{M}$ in characteristic 71}
Consider the Monster group $\text{M}$ which is the sporadic finite simple group with maximal order. The order of $\text{M}$ is approximately $8\times 10^{53}$. The prime $71$ strictly divides the order of $\text{M}$, the group $\text{M}$ contains a subgroup $H$ isomorphic to $\text{PSL}_2(71)$, and the normalizer of a Sylow $71$-subgroup is isomorphic to $\mathbb{Z}/71 \rtimes \mathbb{Z}/35$ \cite[Theorem~1]{MR2397402}. To verify Conjecture~\ref{AIC} for $M$ in characteristic $71$ we must show for every subgroup $I$ of $M$ isomorphic to one of $\{\mathbb{Z}/71, \mathbb{Z}/71 \rtimes \mathbb{Z}/5, \mathbb{Z}/71 \rtimes \mathbb{Z}/7, \mathbb{Z}/71 \rtimes \mathbb{Z}/35 \}$ there exists an $(M,I)$-Galois cover.

Pick $I\in \mathcal{I}_{71}(M)$ and denote the unique Sylow $71$-subgroup of $I$ by $S$. There exists a subgroup $H \cong \text{PSL}_2(71)$ containing $I$. By \cite[Corollary~2.4]{MR2003452}, there exists an $(H, I)$-Galois cover $\phi$. There exists an $(\text{M},S)$-Galois cover $\psi$ \cite[Theorem~2]{MR1221836}. From $\phi$ and $\psi$ Theorem~\ref{thm:FromSubgroups} constructs an $(\text{M}, I)$-Galois cover.

\section{Occurrence of all but Finitely Many Jumps}\label{sec:infConductors}
We now put aside the question of whether there exists a $(G,I)$-Galois cover for every $I \in \mathcal{I}_p(G)$ and instead consider which ramification invariants occur for unramified $G$-Galois covers of $\mathbb{A}^1_k$. Studying which ramification invariants occur loses information concerning the centralizers of the inertia groups which occur. This is not a strict loss, as we gain information regarding which inertia jumps occur. In particular, we realize all but finitely many of the potential ramification invariants for the sporadic groups in Tables \ref{tbl:aicgroups}, \ref{tbl:jumpsgroups1}, and \ref{tbl:jumpsgroups2}.

Fix a prime $p$, finite quasi-$p$ group $G$ with order strictly divisible by $p$, and $k=\overline{\mathbb{F}}_p$. Recall from Section~\ref{sec:ramgrps} that if $p$ strictly divides $G$, then every $I \in \mathcal{I}_p(G)$ must be of the form $I \cong \mathbb{Z}/p\rtimes\mathbb{Z}/m_I$ for some $m_I$ such that $\gcd(p,m_I)=1$. For such a $(G,I)$-Galois cover, the ramification invariant $\sigma$ is related to the inertia jump $j$ by $\sigma = \frac{j}{m_I}$. 

\begin{definition}\label{defn:potentialJumps}
With the above notation, denote the set of potential ramification invariants for a $(G,I)$-Galois cover by
\begin{align*}
  \sigma_p(I) &= \left\{\frac{j}{m_I} \in \mathbb{Q} \mid j > m_I \text{, } p\nmid j \text{, and } \gcd(j,m_I)= \frac{|\text{Cent}(I)|}{p}  \right\}.
\end{align*}

Now let $I$ vary through all $\mathcal{I}_p(G)$ and denote the set of all possible ramification invariants of $(G,I)$-Galois covers in the following way
\begin{align*}
  \sigma_p(G)   &= \bigcup_{I\in \mathcal{I}_p(G)} \sigma_p(I).
\end{align*}
\end{definition}

\begin{definition}
We say ``all but finitely many ramification invariants occur for $G$ in characteristic $p$'' if for all but finitely many $\sigma \in \sigma_p(G)$ there exists a $(G,I)$-Galois cover with ramification invariants $\sigma$ for some $I \in \mathcal{I}_p(G)$.
\end{definition}

\begin{lemma}
Suppose $I\in \mathcal{I}_p(G)$. If for every $\overline{j} \in \mathbb{Z}/m_I$ satisfying $\gcd(\overline{j},m_I) = \frac{|\text{Cent}(I)|}{p}$ there exists a $(G,I)$-Galois cover with ramification invariant $\frac{j}{m_I}$ for some $j \equiv \overline{j} \bmod m_I$, then all but finitely many $\sigma \in \sigma_p(I)$ occur for $(G,I)$-Galois covers.
\end{lemma}
\begin{proof}
In \cite[Lemma~3.2.3]{MR2016596} it is shown that if the inertia jump $j$ occurs for a $(G,I)$-Galois cover, then any $j' > j$ such that $j' \equiv j \bmod m_I$ occurs for some $(G,I)$-Galois cover. Consequently, if there exists a $(G,I)$-Galois cover with inertia jump $j \equiv \overline{j} \bmod m_I$ for each equivalence class $ \overline{j} \in \mathbb{Z}/m_I$ satisfying $\gcd(\overline{j},m_I) = \frac{|\text{Cent}(I)|}{p}$, then all but possibly a few potential inertia jumps smaller than $j$ occur for that equivalence class. Because $I$ has order strictly divisible by $p$, the jump $j'$ corresponds to the ramification invariant $\frac{j'}{m_I} \in \sigma_p(I)$.
\end{proof}

\begin{prop}\label{prop:infjumps}
Suppose $p$ is prime, $G$ is a finite quasi-$p$ group with order strictly divisible by $p$, $S\in \text{\em{Syl}}_p(G)$, and $H$ is a subgroup of $G$ for which there exists an $(H, \text{N}_H(S))$-Galois cover. If for all $I\in \mathcal{I}_p(G)$ there exists a finite group $D$ such that $I = I' \times D$ for some $I' \in \mathcal{I}_p(H)$, then all but finitely many ramification invariants $\sigma \in \sigma_p(G)$ occur.
\end{prop}
\begin{proof}
Let $I = \text{N}_H(S)$ and note $I \cong \mathbb{Z}/p \rtimes \mathbb{Z}/m_I$ by the Schur-Zassenhaus Theorem \cite[pg.\ 132]{MR0091275}. Lemma~\ref{lma:makingjumps} and the Different Inertia case of \cite[Corollary~2.3.1]{MR2016596} show that there exists a $(G,I)$-Galois cover with ramification invariant $\sigma = \frac{\gamma( j + im_I)}{m_I}$ where $j$ and $\gamma$ are coprime to $m_I$.

Pick an element $\overline{j} \in \mathbb{Z}/m_I$ where $\gcd(\overline{j},p)=\frac{|\text{Cent}(I)|}{p}$. There exists a positive integer $d \in \mathbb{N}$ such that $d\gamma j \equiv \overline{j} \bmod m_I$ and $$\frac{d\gamma(j + im_I)}{\gcd(m_I,d)} \equiv \overline{j} \bmod m_I.$$
Let $I' \subset I$ be the subgroup with order $\frac{pm_I}{\gcd(m_I,d)}$. Applying \cite[Proposition~3.1]{MR2003452} yields a $(G, I')$-Galois cover with inertia jump $j' = \frac{d\gamma(j + i m_I)}{\gcd(m_I,d)}$ and ramification invariant $\sigma = \frac{j'}{m_{I'}}$.
\end{proof}

\begin{remark}
Assume the notation of Proposition~\ref{prop:infjumps}. If $D$ is trivial, then all but finitely many $\sigma \in \sigma_p(G)$ occuring is equivalent to Conjecture~\ref{AIC} being true for $G$ in characteristic $p$.
\end{remark}

\begin{definition}
By $m_G$ we will denote the smallest integer such that $m_G \cdot \sigma_p(G) \subset \mathbb{Z}$.
\end{definition}

\begin{thm}\label{thm:infjumps}
As a result of Proposition~\ref{prop:infjumps}, we can verify the occurrence of all but finitely many $\sigma \in \sigma_p(G)$ for the groups and{} characteristics in Table~\ref{tbl:aicgroups} as well as the groups and characterstics in Table~\ref{tbl:jumpsgroups1} and Table~$\ref{tbl:jumpsgroups2}$.
\begin{table}[!h]
\small
\centering
\begin{tabular}{|l | l | l | l | l | l | l | l | l |}
\hline
$p=5$ &               &         & & & $p=7$           &                                      &        & \\  \hline
$G$   & $N_G(S)$      & $m_G$   & $H$ &&   $G$           & $N_G(S)$                             & $m_G$  & $H$\\ \hline 
$\text{J}_1$ & $D_5 \times S_3$ & 2 & $\text{PSL}_2(11)$ & & $\text{M}_{23}$ & $(\mathbb{Z}/7 \rtimes \mathbb{Z}/3) \times \mathbb{Z}/2$ & 3& $\text{PSL}_2(7)$ \\ \hline
$\text{J}_3$ &$D_5 \times S_3$  & 2 & $\text{PSL}_2(19)$ & & $\text{M}_{24}$ & $(\mathbb{Z}/7 \rtimes \mathbb{Z}/3) \times \text{S}_3$ & 3 & $\text{PSL}_2(7)$ \\ \hline
& & & & & $\text{McL}$    & $(\mathbb{Z}/7 \rtimes \mathbb{Z}/3) \times \mathbb{Z}/2$ & 3 & $\text{PSL}_2(7)$\\ \hline
& & & & & $\text{Ru}$    & $D_7 \rtimes A_4$ & 6 & $\text{PSL}_2(13)$\\ \hline
\end{tabular}
\caption{Groups in characteristics 5 and 7 for which all but finitely many jumps are verified along with structure of the normalizer of $S\in\text{Syl}_p(G)$, the value of $m_G$, and the subgroup $H$ for which Proposition~\ref{prop:infjumps} is applied.}
\label{tbl:jumpsgroups1}
\end{table}

\begin{table}[!h]
\small
\centering
\begin{tabular}{| l | l | l | l |}
\hline
$p=11$          &         &                                                       & \\  \hline
$G$             & $N_G(S)$ & $m_G$                                                & $H$ \\ \hline 
$\text{Co}_{3}$ & $(\mathbb{Z}/11 \rtimes \mathbb{Z}/5) \times \mathbb{Z}/2$ & 5  & $\text{PSL}_2(11)$\\ \hline
$\text{F}_{22}$ & $(\mathbb{Z}/11 \rtimes \mathbb{Z}/5) \times \mathbb{Z}/2$ & 5  & $\text{PSL}_2(11)$\\ \hline
\end{tabular}
\caption{Groups in characteristic 11 for which all but finitely many jumps are verified along with structure of the normalizer of $S\in\text{Syl}_p(G)$, the value of $m_G$, and the subgroup $H$ for which Proposition~\ref{prop:infjumps} is applied.}
\label{tbl:jumpsgroups2}
\end{table}
\end{thm}
\begin{proof}
All groups in Tables \ref{tbl:aicgroups}, \ref{tbl:jumpsgroups1} and \ref{tbl:jumpsgroups2} satisfy the hypotheses of Proposition~\ref{prop:infjumps}. In the cases $H\cong \text{PSL}_2(p)$ see \cite[Corollary 3.3]{MR2003452}. For all other cases see \cite[Theorem 3.6]{MR2003452}.
\end{proof}

\section{A Refinement for $\text{M}_{11}$ in characteristic 11}\label{sec:mathieu11}
We realize improved lower bounds on the ramification invariants for $(\M, I)$-Galois covers in characteristic $11$. Specifically all but eight of the possible ramification invariants are shown to occur. We prove Theorem~\ref{thm:thethm} in the following way. Lemma~\ref{lma:possibleRI} describes the possible minimal ramification invariants for an unramified $\M$-Galois cover of $\Aff_k^1$. Lemma~\ref{lma:jumpsToGenera} determines the genera of a quotient cover given a ramification invariant. Then to show that $\sigma=8/5$ occurs with inertia group isomorphic to $\ZZ/11 \rtimes \ZZ/5$, Proposition~\ref{prop:j8} studies a cover in characteristic 11 provided by Serre in \cite{MR1162313}. To show that $\sigma = 2$ occurs with inertia group isomorphic to $\ZZ/11$, Proposition~\ref{prop:j10} studies the semi-stable reduction of a characteristic 0 cover to characteristic 11. Finally, the larger ramification invariants are shown to occur via results of \cite{MR2250044}.

The techniques in this section depend on the $p$-purity of $\text{M}_{11}$ and existence of a proper quasi-$p$ subgroup of sufficiently small index relative to the size of $p$.

\subsection{Intermediate genus formula}\label{sec:intermediateGenus}
Let $G$ be a finite simple group and let $C = (C_1, C_2, C_3)$ be a triple of conjugacy classes in $G$ rational over a field $L$ such that
\begin{align*}
\{(g_1,g_2,g_3) \in C : g_i \in C_i\text{, } g_i \not=1\text{, and } g_1g_2g_3 = 1\} \not= \oslash.
\end{align*}
Assume $\text{char}(L)\nmid|C_i|$. For such a triple, there exists a tame $G$-Galois cover $Y \to \mathbb{P}_L^1$ branched at three points labeled $P_1$, $P_2$, $P_3$ over which an inertia group is generated by some $g_i \in C_i$.

Fix a subgroup $H\subset G$ and let $X=Y/H$. Consider the $H$-Galois subcover $Y\to X$ and degree $[G:H]$ cover $X \to \Proj^1$. Denote the normalizer of $H$ in $G$ by $N_G(H)$ and the inertia group at a point above $P_i$ by $I_i$.
\begin{lemma}
Consider $G$, $H$, $X$, and $Y$ as above. The genus $g$ of $X$ can be computed as follows
\begin{equation}\label{eqn:intermediateGenusFormula}
g = -[G:H] + 1 + \frac{[G:H]}{2}\sum_{i=1}^{3}\frac{|I_i|-1}{|I_i|} - \frac{[N_G(H):H]}{2}\sum_{i=1}^{3}\frac{|N_G(I_i)|}{|N_H(I_i)|}\frac{|H\cap I_i|-1}{|H\cap I_i|}.
\end{equation}
\end{lemma}
\begin{proof}

Write the Riemann-Hurwitz Formulas for the covers $Y\to\Proj_L^1$ and $Y\to X$:
\begin{equation} \label{eq:RH1}
2\genus(Y) -2 = |G|(2\genus(\Proj_L^1)-2) + |G|\sum_{i=1}^{3}\frac{1}{|I_i|}(|I_i|-1);
\end{equation}
\begin{equation} \label{eq:RH2}
2\genus(Y) -2 = |H|(2g-2) + |N_G(H)|\sum_{i=1}^{3}\frac{|N_G(I_i)|}{|N_H(I_i)|}\frac{|H\cap I_i|-1}{|H\cap I_i|}.
\end{equation}
Solving this system of equations for $g$ yields \eqref{eqn:intermediateGenusFormula}.
\end{proof}

\subsection{Vanishing cycles}\label{sec:vanishingcycles}
Let $\phi \colon Y_0\to(X_0 = \Proj_K^1)$ be a $G$-Galois cover defined over a complete discrete valuation field $K$ branched at 0, 1, and $\infty$. Assume the characteristic of the residue field $k$ is $p > 0$ and $p$ strictly divides $|G|$. To force bad reduction, assume that $p$ divides the order of the inertia group at some ramified point. Then $\phi$ has a stable reduction $\phi_s \colon Y_s \to Z_s$ with the following properties \cite[Theorem~2]{MR1992833}.
\begin{itemize}
  \item The base $Z_s$ is a tree of projective lines.
  \item There is a unique original component, denoted $Z$, which each other component of $Z_s$ intersects.
\end{itemize}
The components of $Z_s$ other than Z are called tails. The restriction $\phi_\alpha$ of $\phi_s$ to a tail $X_\alpha$ is a cover of $\Proj_k^1$. The point on a tail $X_\alpha$ where it intersects the original component is called $\infty_\alpha$. A tail cover $X_\alpha$ is called a new tail if it is only ramified at $\infty_\alpha$. Let $P_i$ be the point of $Z_s$ to which $i=0,1,\infty$ specializes. A tail $X_\alpha$ is called a primitive tail if one of the original branch points specializes to it. If $G$ is $p$-pure then the cover is connected over one tail \cite[Proposition~3.1.7]{MR1670532}.

Let $\BB$ be the index set of tails. Each $\alpha \in \BB$ uniquely identifies a tail cover $\phi_\alpha$ and $\sigma_\alpha$ denotes the ramification invariant at $\infty_\alpha$. Let $\BB_\text{new}$ be the index set of new tails, and $\BB_0$ the index set of primitive tails. When all inertia groups have order divisible by $p$, there are no primitive tails.

For $|\BB_0| = 3$, the vanishing cycles formula in \cite[Section 3.4.4]{MR1670532} yields the following.
\begin{equation}\label{eqn:rvc}
\sum_{\alpha \in \BB_\text{new}} (\sigma_\alpha-1) = 1.
\end{equation}

\subsection{Realizing small jumps for $\text{M}_{11}$ in characteristic 11}
Recall from Section \ref{sec:ramgrps} that the inertia group at Q is isomorphic to $\ZZ/11 \rtimes \ZZ/m_I$ where $\gcd(11,m)=1$. In $\M$ the normalizer of a subgroup isomorphic to $\ZZ/11$ is of the form $\ZZ/11 \rtimes \ZZ/5$; thus $m_I=5$ or $m_I = 1$.
\begin{lemma}\label{lma:possibleRI}
There exists an $\text{M}_{11}$-Galois cover $Y\to\Proj_k^1$, only branched at $\infty$, with ramification invariant $\sigma$ is in the set $\{ \frac{6}{5}, \frac{7}{5}, \frac{8}{5}, \frac{9}{5}, 2 \}$.
\end{lemma}
\begin{proof}
 Recall $\M$ is quasi-11, and $\M$ is 11-pure, applying \cite[Theorem~3.5]{MR1937117} proves that a minimal cover exists such that $\sigma \in \{ \frac{6}{5}, \frac{7}{5}, \frac{8}{5}, \frac{9}{5}, 2 \}$.
\end{proof}

Note that this does not solve the inertia conjecture because it does not show that all possible inertia groups occur. The first four ramification invariants are associated to inertia groups isomorphic to $\ZZ/11 \rtimes \ZZ/5$ while $\sigma = 2$ is associated to inertia groups isomorpic to $\ZZ/11$.

To apply results of \cite{MR2777805}, it is important to know the possible degrees of non-Galois covers dominated by an $\M$-Galois cover.
\begin{lemma}\label{lma:degrees}
Let $L$ be an algebraically closed field of any characteristic. Let $X\to\Proj_L^1$ be a degree $d$ non-Galois cover with $\M$-Galois closure $Y\to\Proj_L^1$. If $11\leq d < 22$, then $d\in\{11,12\}$.
\end{lemma}
\begin{proof}
The possible degrees of $X\to\Proj_L^1$ correspond to indices of subgroups $H \subset \M$. The only maximal subgroups with an index in the given range are isomorphic to $\text{M}_{10}$ and $\text{PSL}_2(11)$. In particular $[\M:\text{M}_{10}] = 11$ and $[\M:\PSL]  = 12$.
Any other possible degrees must arise from subgroups of $\text{M}_{10}$ or $\PSL$. The only other candidate subgroup is $\text{A}_6 \trianglelefteq \text{M}_{10}$ which has index $22$ in $G$. Consequently $d\in\{11,12\}$.
\end{proof}

\begin{lemma}\label{lma:jumpsToGenera}
Fix an $(\M,I)$-Galois cover $Y \to \Proj_k^1$ with ramification invariant $\sigma = \frac{j}{5}$. Let $\varphi \colon X \to \Proj_k^1$ be a degree $11\leq d < 22$ quotient cover of $Y$. Let $g = \text{genus(X)}$. If $d=11$, then $g = j-5$ and if $d=12$, then $g = j-6$.
\end{lemma}
\begin{proof}
Pick $\theta \in I$ satisfying $|\theta| =  5$. The number of orbits of $\theta$ acting on $\{p+1, \ldots,d \}$ is denoted by $t$. By \cite[Proposition~1.3]{MR2777805}, $t=\#\varphi^{-1}(\infty) - 1$ and the genus of $X$ is given by $(2j-t-d+1)/2$.
By Lemma~\ref{lma:degrees}, the two possible degrees for $X\to\mathbb{P}_k^1$ are $11$ and $12$. If $d=11$ then $t=0$. Otherwise, $1 \leq t \leq d-p$. Thus when $d=12$ then $t=1$.
\end{proof}

\begin{prop}\label{prop:j8}
There exists an $(\M,\ZZ/11\rtimes\ZZ/5)$-Galois cover with ramification invariant $\sigma = 8/5$.
\end{prop}
\begin{proof}
The curve $C \colon X^{11} + 2X^9 + 3X^8 -T^8$ is an unramified cover of $\mathbb{A}_k^1$ mapping $(X,T) \mapsto T$. It is wildly ramified over $\infty$ with Galois closure $M_{11}$ {\cite[pg.\ 43]{MR1162313}}. Note that this curve has non-ordinary singularities. The geometric genus $3$ can be computed in a computer package such as Magma or Sage. Because $C$ is a degree 11 cover of $\Proj_k^1$, wildly ramified above $\infty$, Lemma~\ref{lma:jumpsToGenera} implies $\sigma = \frac{8}{5}$. The inertia group for a wildly ramified point over $\infty$ with $\sigma=\frac{8}{5}$ is isomorphic to $\mathbb{Z}/11 \rtimes \mathbb{Z}/5$.
\end{proof}
\begin{prop}\label{prop:j10}
There exists an $(\M,\ZZ/11)$-Galois cover with ramification invariant $\sigma = 2$.
\end{prop}
\begin{proof}
Let $C = (C_1, C_2, C_3)$ where each $C_i$ is an 11-conjugacy class of $\M$ and for some $i$ and $j$, $C_i \not= C_j$. Each $C_i$ is rational over $\QQ(\sqrt{-11})$; let $L = \QQ(\sqrt{-11})$.
Consider an $\M$-Galois cover $Y_0\to \Proj_L^1$ branched at three points $P_1$, $P_2$, and $P_3$ with an inertia group over $P_i$ generated by some element of $C_i$. Also consider the degree $12$ quotient cover $X_0 \to \Proj_L^1$ dominated by the $\PSL$-Galois cover $Y_0 \to X_0$. Applying \eqref{eqn:intermediateGenusFormula} with $C$ and $d=12$ yields $\text{genus}(X_0) = 4$.
\begin{table}[htbp]
\centering
\small
  \begin{tabular}{ | l | c | r |}
    \hline
    $|\mathbb{B}_\text{new}|$ & $\{ \sigma_\alpha : \alpha \in \mathbb{B}_\text{new} \}$ & $\sum_{\alpha \in \BB_{\text{new}}} \text{genus}(X_\alpha)$\\[2pt] \hline &&\\[-1em]
    1 & $\{ \frac{10}{5} \}$ & 4 \\[2pt] \hline &&\\[-1em]
    2 & $\{ \frac{6}{5}, \frac{9}{5} \}$ or $\{ \frac{7}{5}, \frac{8}{5} \}$ & 3 \\[2pt] \hline &&\\[-1em]
    3 & $\{ \frac{6}{5}, \frac{6}{5}, \frac{8}{5} \}$ or $\{ \frac{6}{5}, \frac{7}{5}, \frac{7}{5} \}$ & 2 \\[2pt] \hline &&\\[-1em]
    4 & $\{ \frac{6}{5}, \frac{6}{5}, \frac{6}{5}, \frac{7}{5} \}$ & 1\\[2pt] \hline &&\\[-1em]
    5 & $\{ \frac{6}{5}, \frac{6}{5}, \frac{6}{5}, \frac{6}{5}, \frac{6}{5} \}$ & 0 \\[2pt]
    \hline
  \end{tabular}
  \caption{Possible genera for the reduction of $X\to\Proj^1$ of degree $11$.}
  \label{tbl:raynaud2}
\end{table}

The vanishing cycles formula \eqref{eqn:rvc} gives a set of possibilities for $\{\sigma_\alpha : \alpha \in \BB_{\text{new}} \}$. For the selected ramification type, $|\BB_0| = 3$. Because all $C_i$ are conjugacy classes of order 11, none of the tails indexed by $\BB_0$ are primitive. Thus the vanishing cycles formula is
\begin{equation}\label{eqn:specificRVC}
\sum_{\alpha \in \BB_\text{new}} (j_\alpha/5-1) = 1.
\end{equation}
From \cite[Proposition~3.3.5]{MR1670532}, note that $5<j_\alpha$. For each set of possible ramification invariants use $(2j-t-d+1)/2$ to compute the sum of the genera of the curves $X_\alpha$.

Because $Y_0$ dominates a genus 4 cover, its reduction must as well. This only occurs in the first row of Table \ref{tbl:raynaud2} for the single new tail with ramification invariant $2$. The 11-purity of $\M$ ensures that the cover is connected over the tail component. Thus $\sigma = 2$ occurs with inertia group isomorphic to $\ZZ/11$.
\end{proof}

\begin{thm}\label{thm:thethm}
Abhyankar's Inertia Conjecture is true for $\M$ in characteristic $p=11$. More generally:
\begin{enumerate}[label=\alph*)]
  \item If $j \in \{ 8+i5, 16+i5, 24+i, 32+i55 \mid i\in\ZZ_{\geq 0} \}$ and $p\nmid j$, then $\sigma = j/5$ occurs as a ramification invariant for an $\M$-Galois cover of $\Proj_k^1$ branched at a single point and with inertia groups isomorphic to $\ZZ/11 \rtimes \ZZ/5$.
  \item If $\sigma \in \{2+i \mid i\in\ZZ_{\geq 0} \}$ and $p\nmid \sigma$, then $\sigma = 2+i$ occurs as a ramification invariant for an $\M$-Galois cover of $\Proj_k^1$ branched at a single point and with inertia groups isomorphic to $\ZZ/11$.
\end{enumerate}
\end{thm}
\begin{proof}
Recall that the only possible inertia groups for an $\M$-Galois cover of $\mathbb{A}_k^1$ are isomorphic to $\ZZ/11 \rtimes \ZZ/5$ and $\ZZ/11$. By Propositions~\ref{prop:j8} and \ref{prop:j10}, each of these occurs with ramification invariants $8/5$ and $2$ respectively. The other inertia jumps can be produced with applications of \cite[Corollary~2.3.1 Different Inertia Case]{MR2016596} with $r=1$. To see that $j = 16$ occurs, apply Theorem~\ref{thm:FromSubgroups} with $G_1 \cong G_2 \cong \M$, $I_1 \cong I_2 \cong \ZZ/11 \rtimes \ZZ/5$, and $j_1 = j_2 = 8$. Theorem~\ref{thm:FromSubgroups} can be reapplied with $j_1 = 16$ yielding $j = 24$. Likewise applying Theorem~\ref{thm:FromSubgroups} a final time with $j_1 = 24$ produces $j = 32$.

Finally \cite[Theorem~3.2]{MR2250044} allows $j$ to be increased by multiples of $5$.
\end{proof}

This method is not sufficient to determine whether these jumps $j$ occur: 6, 7, 9, 12, 14, 17, 19, and 27.

\printbibliography

\Addresses

\end{document}